\documentclass{article}
\usepackage{graphicx} 
\usepackage{amsmath}
\usepackage{amssymb}
\usepackage{amsthm}
\usepackage[utf8]{inputenc}
\usepackage{hyperref}
\usepackage{geometry}
\usepackage{cite}
\usepackage{authblk}
\usepackage{caption}
\numberwithin{equation}{subsection}

\author{Zihan Cui  \thanks{ Department of Statistics, University of Michigan. Email: childhan@umich.edu}}

\newtheorem{theorem}{Theorem}
\newtheorem{lemma}{Lemma}
\newtheorem{corollary}{Corollary}

\title{\textbf{Early Stopping in Contextual Bandits and Inferences}}

\date{}

\begin{document}

\maketitle

\begin{abstract}
\noindent Bandit algorithms sequentially accumulate data using adaptive sampling policies, offering flexibility for real-world applications. However, excessive sampling can be costly, motivating the devolopment of early stopping methods and reliable post-experiment conditional inferences. This paper studies early stopping methods in linear contextual bandits, including both pre-determined and online stopping rules, to minimize in-experiment regrets while accounting for sampling costs. We propose stopping rules based on the Opportunity Cost and Threshold Method, utilizing the variances of unbiased or consistent online estimators to quantify the upper regret bounds of learned optimal policy. The study focuses on batched settings for stability, selecting a weighed combination of batched estimators as the online estimator and deriving its asymptotic distribution. Online statistical inferences are performed based on the selected estimator, conditional on the realized stopping time. Our proposed method provides a systematic approach to minimize in-experiment regret and conduct robust post-experiment inferences, facilitating decision-making in future applications.
\end{abstract}


\section{Introduction}
Bandit Algorithms have been increasingly popular in sequential data analysis and are widely used in a range of industries like clinical trials \cite{phase3randomized}, online advertising \cite{anagnostopoulos2007just} and recommendation systems \cite{li2010contextual}. Compared to traditional statistical models and methods, where data is accumulated all at one time, bandit algorithms have data accumulated sequentially and require us to select adaptive sampling policies in each time step. This fact gives bandit algorithms many advantages over traditional methods in various ways. First, bandit algorithms enable us to manipulate the trade-off between exploitation and exploration. Based on previous results, we can take greedy policies and choose to sample more on those arms with more desirable performances, which speeds up the exploitation process. To avoid the possible situation that the greedy policy may be trapped into some actions that have good performances in the early stage and leave other actions unexplored \cite{chen2021statistical}, we can introduce random policies to guarantee a sufficient amount of exploration. Using a mixture of greedy and random policies, we can manipulate the rate of exploitation and exploration process. Moreover, sequential sampling is more efficient because we are able to stop the experiment once we observe that the outcome is good enough for us to draw a conclusion or bad enough so that we do not want to waste more resources on the experiment. In fact, acquiring information through experiments could sometimes be costly, so experimenters need to carefully choose how many units of each treatment to sample and when to stop sampling \cite{adusumilli2022sample}. Sequential sampling makes early stopping possible and is thus more efficient and flexible.

Although bandit algorithms have many advantages over traditional statistical models, the fact that data is sampled according to some adaptive policies dependent on previous outcomes brings multiple challenges to our analysis. One challenge is that the distribution of variables in the experiment become more complicated and the relations between them are more inexplicit. Previous observations can influence future policies, thereby indirectly affecting future observations. As a result, it is necessary to consider the conditional distribution of online data on previous history rather than treating them as independent variables. In addition, adaptive algorithms and online inferences may be sensitive to noise or outliers in the data stream, potentially leading to unstable or incorrect decisions. Therefore, ensuring the robustness of our adaptive experiment is essential.

In this work, we focus on developing early stopping algorithms for linear contextual bandit problems with costly observations. Our goal is to minimize regret in the experiment while taking sampling costs into account. We explore both pre-determined stopping rules and online stopping rules for this goal. To devise the pre-determined stopping rules, we first represent the regret of taking the learned optimal policy as a function of time. This can be done by evaluating the tail bound of the OLS estimators using concentration inequalities. To balance the trade-off between the regret and sampling costs, we propose our stopping rules based on Opportunity Cost and specified thresholds, which are fully determined by some known information before the experiment starts. 

While pre-determined stopping methods are easy to implement, they come with limitations in flexibility. In contrast, online stopping rules are more flexible in experimentation as they capture the dynamics of the data. To develop online stopping rules, we also quantify the regret in the adaptive experiment, but using online data rather than pre-known information. We propose that if we can identify an unbiased or consistent online estimator of our interests, we can then stop the experiment based on the estimated variance of that estimator. This approach offers an effective way to control the regret in adaptive experiments.

Our another objective is to conduct valid inferences on parameters after the experiment stops. We can regard the information about those parameters some kind of "scientific knowledge". While regret minimization is a within-experiment learning objective, gaining scientific knowledge from the resulting adaptively collected data is a between-experiment learning objective \cite{zhang2021statistical}. The gained scientific knowledge helps us develop a better understanding of the mechanism in similar bandit problems and minimize regret in future experiments. 

However, online statistical inference could be intricate and challenging in many ways. With adaptively collected data, common estimators based on sample means and inverse
propensity-weighted means can be biased or heavy-tailed \cite{hadad2021confidence}. Additionally, the sensitivity of online algorithms can result in unstable outcomes in adaptive experiments. To cope with this, we consider the linear contextual bandit problem in a batched setting for stability. We use a weighted combination of batched OLS estimators as the final online estimator for inferences and derive its asymptotic distribution. Furthermore, the realized stopping time imposes some restrictions on the behavior of the past data stream. In other words, the distribution of online data changes at the moment the experiment stops. This motivates us to consider the distribution of the chosen online estimator conditional on the realized stopping time. We find that if the stopping criterion can be expressed as a function of batched OLS estimators and their variances, the conditional distribution of our selected online estimator would be a truncated Gaussian. Under this scenario, we propose a general conditional inference procedure using Gibbs sampling and demonstrate its application on both simulated and real data.

\subsection{Related Literature}
\textbf{Linear Contextual Bandit.} Our work focuses on the linear contextual bandit setting. Many papers have discussed online algorithms and analyzed their regret bounds under this setting. Chu et al. \cite{chu2011contextual} analyze the linear Upper Confidence Bound algorithm and derive its upper and lower regret bound of order $\tilde{O}(\sqrt{T})$. Chen et al. \cite{chen2021statistical} discuss the $\epsilon$-greedy algorithm and give the tail bound for the online OLS estimator. Shen et al. \cite{shen2024doubly} extend this result and derive the tail bounds for the online OLS estimator using Thompson Sampling and Upper Confidence Bound algorithms.

\textbf{Early Stopping.} A few papers discuss early stopping algorithms in bandit and reinforcement learning problems. Even-Dar et al. \cite{even2006action} analyze the times of pulling the arms it take to find an $\epsilon$-optimal arm with at least $1-\delta$ probability. They also devise action elimination procedures in reinforcement learning algorithms and derive stopping conditions that is approximately optimal. Tucker et al. \cite{tucker2023bandits} use concentration inequalities to derive small-width bounds for true parameters in linear contextual bandits, and devise stopping methods by comparing expected future rewards and observation costs. However, it only focuses on the regret minimization part and does not provide inferences on parameters conditional on the stopping time. 

\textbf{Online Inference.} Online learning inference is intricate, as mentioned earlier. Fortunately, there is a growing body of literature discussing this important topic. Most of them consider the asymptotic distribution of their proposed online estimators. Zhang et al. \cite{zhang2021statistical} utilize adaptively weighted M-Estimators and derive its asymptotic normality and construct uniform confidence regions. Adding the adaptive importance weight terms stabilizes the variances of martingale difference sequences in M-estimators. In another paper, Zhang et al. \cite{zhang2020inference} first state that the asymptotic non-normality of OLS estimator in bandit problems can lead to inflated Type-1 error, and then solve it by introducing Batched OLS estimator and prove its asymptotic normality on data collected from both multi-arm and contextual bandits. However, it is still complicated to be used for inferences because there are multiple Batched OLS estimators. In addition to analyzing the asymptotic distribution of online estimators, some papers also focus on conditional inference. Chen and Andrews \cite{chen2023optimal} consider the Multi-Arm Bandit problem in which assignment probabilities, stopping time and target parameter depend on the history of outcomes through location-invariant functions or a collection of polyhedral events, and derive the optimal conditional inference procedure using sufficient statistics.

\section{Problem Formulation}
\subsection{Linear Contextual Bandit}
For the Linear Contextual Bandit problem, we first consider it under two-arm setting and then generalize it to K-arm situation. At each time step $t=1,2,\cdots,T_0$, we observe context $x_t \in \chi \subset R^d$ sampled from an unknown i.i.d. distribution $P_X$, and take an action $a_t$ from the action space $A=\{0,1\}$ based on both contextual information $x_t$ and adaptive policy $\pi_t$. Here, the selection probability follows $P^{\pi_t}(a_t|x_t)=\pi_t(a_t|x_t)$. After taking an action $a_t=1$ or $0$, we will observe a reward $y_t$ with respect to the corresponding arm. The observed reward $y_t$ has a linear form 
$$y_t=a_tx_t^T\beta_1+(1-a_1)x_t^T\beta_0+e_t$$
$$E(e_t|a_t)=0, E(e_t^2|a_t)=\sigma^2$$
$$e_t\perp x_t|a_t$$

Let $\mathcal{F}_t=\sigma(x_1,a_1,y_1,\cdots,x_t,a_t,y_t)$ be the sigma field generated by the history up to time t. At the end of time step t, we update our policy from $\pi_t$ to $\pi_{t+1}$ using the history $\mathcal{F}_t$, i.e., $\pi_{t+1} \in \mathcal{F}_t$. The Ordinary Least Square estimator $\hat{\beta}_{t,i}$ for $\beta_i$ at the end of each time step is 
$$\hat{\beta}_{t,i}=[\sum_{j=1}^t I(a_j=i)x_j x_j^T ]^{-1} \sum_{j=1}^t I(a_j=i)x_j y_j^T, i=0,1.$$

The selection of $\{\pi_t\}_{t\geq 1}$ vary a lot in different bandit algorithms. The most commonly used algorithms are the $\epsilon$-greedy algorithm, Upper Confidence Bound and Thompson Sampling. 

\textbf{$\epsilon$-greedy (Sutton and Barto \cite{sutton2018reinforcement}):} Suppose we use a non-increasing sequence $p_t$ to control the exploration-exploitation process, then at time t we choose the more desirable arm $I(x_t^T \hat{\beta}_{t,1}>x_t^T \hat{\beta}_{t,0})$ with probability $1-\frac{p_t}{2}$ and the other arm $I(x_t^T \hat{\beta}_{t,1}\leq  x_t^T \hat{\beta}_{t,0})$ with probability $\frac{p_t}{2}.$

\textbf{Upper Confidence Bound (UCB) (Li et al., 2010 \cite{li2010contextual}):} Let the 
estimated standard deviation based on $\mathcal{F}_{t-1}$ be $\hat{\sigma}_{t-1}(x,i) = \sqrt{ x^T [\sum_{j=1}^{t-1} I(a_j=i)x_j x_j^T ]^{-1} x }, i=0,1 $. We select the action at time t by $$ a_t=\arg\max\limits_{i \in A} \hat{\beta}_{t,i}+c_t \hat{\sigma}_{t-1}(x,i) ,i=0,1 ,$$
where $c_t$ is a non-increasing positive sequence.

\textbf{Thompson Sampling (Agrawal and Goyal \cite{agrawal2013thompson}):} Suppose that the error term satisfies a normal distribution $e_t|a_t \sim N(0,\sigma^2)$ with known $\sigma^2$, and the prior for $\beta_i, i=0,1$ at time t is $$ \beta_i \sim N(\hat{\beta}_{t-1,i},[\sum_{j=1}^{t-1} I(a_j=i)x_j x_j^T ]^{-1} \sigma^2  )  ,$$then we have the posterior for $\beta_i$ being
$$\beta_i|\mathcal{F}_t  \sim N(\hat{\beta}_{t,i},[\sum_{j=1}^{t} I(a_j=i)x_j x_j^T ]^{-1} \sigma^2  ), i=0,1.$$We draw sample $\beta_t(0), \beta_t(1)$ from the posterior distributions of $\beta_0,\beta_1$ respectively, and choose the action $a_{t+1}=I(x_{t+1}^T \beta_t(1)> x_{t+1}^T \beta_t(0)  ) $ at time t+1.

In the linear contextual bandit setting, certain assumptions are required to ensure the validity of our OLS estimator. First, to prevent the inverse term in the OLS estimators from blowing up, we impose the following two assumptions on the distribution of $x_t$:

\textbf{Assumption 1.} The context should be bounded in $R^d$, i.e., there exists $L>0$ such that $||x||_{\infty} \leq L$ for $x \in \chi.$

\textbf{Assumption 2.} For the matrix $\Sigma= E_{x\sim P_X} (xx^T) $, there exists some $q>0$, such that the minimum eigenvalue of $\Sigma $ is greater than q, i.e., $\lambda_{\min}(\Sigma) > q.$

In addition, a margin condition for differences between arms is needed to ensure the convergence of estimators. In the two-arm situation, we have the following assumption:

\textbf{Assumption 3.} There exists a uniform $M>0$ such that for every $h>0$, we have $P(|(\beta_1-\beta_0)^Tx| \leq h) \leq M h^{\lambda},$ where $\lambda>0$ is a constant.

To ensure the stability of the online algorithms, we consider the linear contextual bandit problem in a batched setting. Under the batched setting, we have time rounds $t=1,2,\cdots,T_0$. In each time round t, we observe $n_t$ contexts $x_{t,1},x_{t,2},\cdots,x_{t,n_t}$, take the corresponding actions $a_{t,1},a_{t,2},\cdots,a_{t,n_t}$ based on the policy $\pi_t$ in this round, and observe the rewards $y_{t,1},y_{t,2},\cdots,y_{t,n_t}$. The policy $\pi_t$ we take within each time round remains the same, and we only update it at the end of that round. The linear model for rewards remains the same, and the only difference between non-batched and batched setting is whether we update our policies after observing every single outcome. For simplicity, we assume that the sample sizes $n_t$ in each round are the same, i.e., $n_t=n, t=1,2,\cdots,T_0.$

In order to guarantee enough exploration using our online algorithms, we impose clipping on them.

\textbf{Clipping.} In each batch, we force the online policy to explore all arms by introducing a clipping probability $p_t$ s.t. $p_t \leq P(a_t=i|\pi_t,x_{t,j})\leq 1-p_t, i= 0,1, j=1,\cdots,n$.

\subsection{Early Stopping}
The goal of our research is to devise valid early stopping rules. Generally, as the experiment goes, we obtain more precise information about unknown parameters and achieve lower regret, while sampling costs continue to increase. Therefore, it is necessary to establish the regret, taking sampling costs into account, as a function of time.

Suppose that at the end of time point t, we have learned estimators $\hat{\beta}_{t,1}$ and $ \hat{\beta}_{t,0}$ for $\beta_1, \beta_0$ respectively. The learned optimal policy induced by these estimators is given by $\pi_{t}^*(x)=I (( \hat{\beta}_{t,1}-\hat{\beta}_{t,0})^Tx>0 ).$ Let $\pi^*(x)=I( (\beta_1-\beta_0)^T x>0   )  $ denote the true optimal policy. The regret of following the learned optimal policy $\pi_t^*$ is then defined as $$R_{\pi_t^*}=V(\pi_t^*)-V(\pi^*) ,$$ where $V(\cdot)$ is the policy value function. 

In practice, the exact value of $R_{\pi_t^*}$ is unknown, but it is possible to evaluate an upper bound for it. Let $U(R_{\pi_t^*})$ be an upper bound for $ R_{\pi_t^*}$, then we can define the regret in batched setting, incorporating sampling cost, as
\begin{equation}cReg=U(R_{\pi_t^*})+cnt. \end{equation}
Where c is the unit sampling cost.

Sometimes, we have strict requirements on the upper bound of regret and need to ensure that it remains below a given threshold. In this case, the regret with sampling cost can be represented as
\begin{equation}cReg= \infty*I(U( R_{\pi_t^*})>k  )+cnt. \end{equation}
This formulation enforces the constraint that if the upper bound $U( R_{\pi_t^*}) $ exceeds the threshold k, the regret becomes infinitely large. Once the upper bound falls below the threshold, we treat it as equivalent to zero in our analysis.

After establishing the regret functions that incorporate sampling costs, we state our proposed stopping rules in the next section.

\section{Pre-determined stopping rules and Regret Analysis}
In adaptive experiments, the quantities $||\hat{\beta}_{t,1}-\beta_1||$ and $||\hat{\beta}_{t,0}-\beta_0||$ typically decay exponentially, which can be established using concentration inequalities. Knowing the tail bounds for these estimators, we can derive an upper bound for $R_{\pi_t^*}$, as stated in Theorem 1.
\begin{theorem} If $||\hat{\beta}_{t,1}-\beta_1||\leq B_t$ and $||\hat{\beta}_{t,0}-\beta_0||\leq B_t$ hold with probability at least $1-\delta$, then $R_{\pi_t^*}\leq (2B_tL)^{1+\lambda}M$ with probability at least $1-\delta$.
\end{theorem}

\begin{proof} We can write the regret as $$R_{\pi_t^*}= E [ \lvert(\beta_1-\beta_0)^Tx \rvert *I\{\operatorname{sgn}((\hat{\beta}_{t,1}-\hat{\beta}_{t,0})^Tx) \neq \operatorname{sgn}((\beta_1-\beta_0)^Tx )\} ] $$
$$=\int\limits_{\substack{(\beta_1-\beta_0)^Tx>0\\ (\hat{\beta}_{t,1}-\hat{\beta}_{t,0})^Tx<0} }(\beta_1-\beta_0)^Tx dP_X +\int\limits_{\substack{(\beta_1-\beta_0)^Tx<0\\ (\hat{\beta}_{t,1}-\hat{\beta}_{t,0})^Tx>0} }(\beta_0-\beta_1)^TxdP_X$$
Let $R_1=\int\limits_{\substack{(\beta_1-\beta_0)^Tx>0\\ (\hat{\beta}_{t,1}-\hat{\beta}_{t,0})^Tx<0} }(\beta_1-\beta_0)^Tx dP_X$, $R_2=\int\limits_{\substack{(\beta_1-\beta_0)^Tx<0\\ (\hat{\beta}_{t,1}-\hat{\beta}_{t,0})^Tx>0} }(\beta_0-\beta_1)^TxdP_X$.
For the first term $R_1$, we have
$$R_1=\int\limits_{\substack{(\beta_1-\beta_0)^Tx \geq 2B_tL\\ (\hat{\beta}_{t,1}-\hat{\beta}_{t,0})^Tx<0} }(\beta_1-\beta_0)^Tx dP_X +\int\limits_{\substack{0<(\beta_1-\beta_0)^Tx< 2B_tL\\ (\hat{\beta}_{t,1}-\hat{\beta}_{t,0})^Tx<0} }(\beta_1-\beta_0)^Tx dP_X$$
$$\stackrel{\text{$1-\delta$ }} {=} 0+\int\limits_{\substack{0<(\beta_1-\beta_0)^Tx< 2B_tL\\ (\hat{\beta}_{t,1}-\hat{\beta}_{t,0})^Tx<0} }(\beta_1-\beta_0)^Tx dP_X$$
$$\leq 2B_tLP(0<(\beta_1-\beta_0)^Tx< 2B_tL)   $$

The first term in $R_1$ vanishes because $\lvert (\beta_1-\beta_0)^Tx-(\hat{\beta}_{t,1}-\hat{\beta}_{t,0})^Tx \rvert  \leq \lVert(\hat{\beta}_{t,1}-\beta_1)\rVert*\lVert x \rVert+\lVert(\hat{\beta}_{t,0}-\beta_0)\rVert* \lVert x \rVert \leq 2B_tL$ with probability at least $1-\delta$.

Similarly, we have $R_2\leq 2B_tLP(-2B_tL<(\beta_1-\beta_0)^Tx<0)$ with probability at least $1-\delta$.

Therefore, $R_{\pi_t^*} = R_1+R_2\leq 2B_tL P(\lvert(\beta_1-\beta_0)^Tx \rvert <2B_tL)=(2B_tL)^{1+\lambda}M$ with probability at least $1-\delta$.
\end{proof}

For the tail bound of $\hat{\beta}_{t,1}$, $\hat{\beta}_{t,0}$, we have the following result under non-batched setting from Chen et al. \cite{chen2021statistical}.

\begin{lemma} We have the tail bound for OLS estimators $\hat{\beta}_{t,0}, \hat{\beta}_{t,1}$ as
$$P(\lVert \hat{\beta}_{t,i}-\beta_{t,i}\rVert_1\geq h) \leq C_1e^{-C_2tp_t^2h^2},i=0,1,$$ where $C_1,C_2$ are constants determined by $L,\lambda,\sigma,d$. As a result, we have $$||\hat{\beta}_{t,i}-\beta_{t,i}||_1 \leq \sqrt{\frac{\log(\frac{\delta}{C_1})}{-C_2tp_t^2}}=\sqrt{\frac{K}{tp_t^2}},i=0,1$$ with probability at least $1-\delta$. 
\end{lemma}

By combining Theorem 1 and Lemma 1, we express the upper bound of regret $R_{\pi_t^*}$ as a function of time in both non-batched and batched settings, as stated in Corollary 1 and Corollary 2 respectively.

\begin{corollary} Under non-batched setting, the regret $R_{\pi_t^*}\leq (2L\sqrt{K})^{1+\lambda}M(\sqrt{\frac{1}{tp_t^2}})^{1+\lambda}=K'(\sqrt{\frac{1}{tp_t^2}})^{1+\lambda}.$
\end{corollary}

\begin{corollary} Under batched setting, the regret $R_{\pi_t^*}\leq K'(\sqrt{\frac{1}{ntp_t^2}})^{1+\lambda}$.
\end{corollary}

Next, we propose our pre-determined stopping rules under batched setting. If we consider the upper bound $U(R_{\pi_t^*} ) = K'(\sqrt{\frac{1}{ntp_t^2}})^{1+\lambda}  $, then we can devise our optimal stopping rules based on Opportunity Cost and pre-specified thresholds.

\textbf{Opportunity Cost.}
If the regret function is chosen as (2.2.1), that is, $cReg=U(R_{\pi_t^*})+cnt$, we can construct the optimal stopping rule by analyzing the opportunity cost. If the reduction in regret achieved by continuing sampling does not outweigh the associated sampling costs, we would stop the experiment. Specifically for batched bandits, the optimal stopping rule is: Stop if $  U({\pi_t^*})-U({\pi_{t+1}^*}) \leq cn$, i.e. $K'(\sqrt{\frac{1}{ntp_t^2}})^{1+\lambda}-K'(\sqrt{\frac{1}{n(t+1)p_{t+1}^2}})^{1+\lambda} \leq cn.$

\textit{Discussion.} We can estimate the approximate regret of our stopping rule based on Opportunity Cost in specific cases. For example, suppose that the margin parameter $\lambda=1$, clipping probabilities $p_t \to p>0$, and we treat $p_{t+1}\approx p_{t}\approx p$ approximately when t is large, then $U(R_{\pi_t^*}) = \frac{K'}{ntp^2}=\frac{K''}{t} .$ The optimal stopping rule becomes: $\frac{K'}{np^2}(\frac{1}{t}-\frac{1}{t+1})\leq cn$, that is, stop the experiment when $t \approx \sqrt{\frac{K'}{cn^2p^2} } =\sqrt{\frac{K''}{cn}}.$ This yields a regret of

$$cReg^* = \sum_{t=1}^{t^*} \frac{K''}{t}+cnt^*\approx K'' \ln t^*+cnt^*= \frac{K'}{2n p^2} \ln(\frac{K'}{cn^2 p^2})+\sqrt{\frac{K'c}{p^2}}.$$

\textbf{Threshold.} If the regret function is chosen as (2.2.2), that is, $cReg= \infty*I(U( R_{\pi_t^*})>k  )+cnt$, we can stop the experiment once the upper bound of regret is below the threshold k, that is, we stop when $U({R_{\pi_t^*})}) = K'(\sqrt{\frac{1}{ntp_t^2}})^{1+\lambda} \leq k.$

Again for the above case where $p_{t+1}\approx p_t \approx p, \lambda=1 $, the approximate optimal stopping time is $t \approx \frac{K'}{p^2nk}$, which yields a regret of
$$cReg^* \approx \frac{K'c}{p^2k}.$$

The selection of the threshold is flexible in practice. Generally, it should be neither too large nor too small. A large threshold would make the estimation process highly unreliable, while a small threshold would require running the costly experiment for an extended period to meet the desired level. Therefore, choosing an appropriate threshold involves a trade-off between accuracy and cost. In fact, this decision depends on people's expectations for the outcome of the experiment. For example, in autonomous driving, the regret threshold should be small enough to ensure passenger safety. In contrast, a larger regret threshold may be acceptable in online advertising, where a modest degree of inaccuracy does not lead to severe consequences.

\section{Online Stopping Rules and Inferences}
\subsection{General Principles}
While off-data stopping rules are often designed to guarantee low regret bounds ahead of experiment, they do come with their problems and limitations. First, the regret bound derived under a worst-case scenario without any data may not align with the true regret. Additionally, off-data stopping rules fail to capture the dynamics of the data and the exploration-exploitation process. In fact, prescribed stopping criteria are usually too conservative and prioritize safety, leading to the result that the algorithm may stop either too early or too late. Besides, off-data stopping could bring additional challenges in precision, robustness, and interpretability. Therefore, it is crucial to develop online stopping rules to address these issues.

The most critical aspect of online stopping rules is devising a valid method to quantify the adaptive regret bound, which determines whether to continue sampling or not. While there are various approaches to quantify regret bounds, the variance of online estimators could be one of the simplest and most explicit terms associated with the adaptive bounds. We can stop the experiment when the variances of the online estimators are small enough to ensure a narrow confidence interval with high probability, thereby maintaining a low regret bound.

To derive the relationship between variances of estimators and regret bounds, we consider multidimensional concentration inequalities. Suppose $\beta$ is our parameter of interest, if we can find an unbiased online estimator $\hat{\beta}$, then we can bound the probability that it deviates from $\beta$ with high probability when its covariance matrix is small enough in some sense according to the following multivariate Chebyshev inequality:

\textbf{Multivariate Chebyshev Inequality.} Suppose $\hat{\beta}$ is a d-dimensional random vector, $E(\hat{\beta})=\beta, Cov(\hat{\beta})=V$, then we have the following inequality:$$P(\sqrt{ (\hat{\beta}-\beta)^TV^{-1}(\hat{\beta}-\beta) }\geq h) \leq \frac{d}{h^2} .$$

As a result, we have the following bound for $\hat{\beta}$:

\begin{lemma} With a probability of at least $1-\delta$, we have $\lVert \hat{\beta}-\beta  \rVert  \leq \sqrt{ \frac{d \lVert V  \rVert_2 }{\delta } }.$ \end{lemma}

\begin{proof} The Mahalanobis distance satisfies $(\hat{\beta}-\beta)^TV^{-1}(\hat{\beta}-\beta) \geq \frac{\lVert \hat{\beta}-\beta  \rVert^2 }{\lambda_{\max} (V) }= \frac{ \lVert \hat{\beta}-\beta  \rVert^2}{  \lVert V \rVert_2 }.$ Therefore, we have
$$P(\lVert \hat{\beta}-\beta  \rVert  \leq \sqrt{ \frac{d \lVert V  \rVert_2 }{1-\delta } })\geq P((\hat{\beta}-\beta)^TV^{-1}(\hat{\beta}-\beta) \leq \frac{d}{\delta} ) $$
$$\geq 1-\delta .$$ \end{proof}

\begin{theorem} Suppose that $\hat{\beta}_1$ and $ \hat{\beta}_0$ are unbiased estimators of the true parameters $\beta_1$ and $\beta_0$ respectively, with covariances satisfying $\lVert Cov(\hat{\beta}_1) \rVert_2 \leq k,\lVert Cov(\hat{\beta}_1) \rVert_2 \leq k$ . Let $\pi_{\hat{\beta}_{0,1}}$ be the optimal policy generated by estimators $\hat{\beta}_1$ and $\hat{\beta}_0$, then with a probability of at least $1-\delta$, we have $R_{\pi_{\hat{\beta}_{0,1}}} \leq  M (2L\sqrt{ \frac{dk}{\delta}} )^{1+\lambda}.$
\end{theorem}

\begin{proof} According to Lemma 2, with a probability of at least $1-\delta$, we have $\lVert \hat{\beta}_1-\beta_1  \rVert  \leq \sqrt{ \frac{d k }{\delta } },\lVert \hat{\beta}_0-\beta_0  \rVert  \leq \sqrt{ \frac{d k }{\delta } }$. Using Theorem 1, we get $R_{\pi_{\hat{\beta}_{0,1}}} \leq ( 2L\sqrt{ \frac{dk}{\delta} }   )^{1+\lambda}M .$\end{proof}

From Theorem 2 we see that the regret upper bound is a known increasing function of the spectral norm of the covariance matrix. In other words, the covariance matrix directly quantifies the upper bound of the regret. Therefore, off-data stopping rules focused on minimizing regret upper bounds can be adapted for online stopping rules based on the variances of unbiased estimators. Also, note that the different matrix norms in finite dimension are equivalent, so it does not matter which norm to choose when we devise stopping rules based on variances. As a result, we have the following online stopping rules using the threshold and opportunity cost methods.

\textbf{Online Stopping-Threshold.} If $\beta$ is our parameter of interest, $\hat{\beta}_t$ is an unbiased estimator of $\beta$, suppose $S_t$ is an estimate of the variance of $\hat{\beta}_t $ using online data up to batch t, then we stop the experiment if the norm of $S_t$ is below a threshold k, i.e., stop when $ \lVert S_t \rVert \leq k$. 

When k is small, this stopping rule guarantees a low regret upper bound on $R_{\pi_{\hat{\beta}_t}^*}$.

\textbf{Online Stopping-Opportunity Cost.} Suppose that $\hat{\beta}_t$ is an unbiased estimator of $\beta$, with estimated variances $S_t$, both of which use online data up to batch t; $\hat{\beta}_{t-1}$ is the same unbiased estimator of $\beta$ but uses data up to batch $t-1$, with estimated variances $S_{t-1}$. If $\lVert S_{t-1} \rVert -\lVert S_t  \rVert \leq c'n$, we stop the experiment at the end of batch t. Here, $c'$ is the transformed opportunity cost in terms of variances, which is different from sampling costs c.

In fact, this opportunity cost stopping rule in terms of variances is equivalent to the stopping rule in terms of regret upper bounds, that is, $U(R_{\pi_{\hat{\beta}_{t-1}}^*}) -U(R_{\pi_{\hat{\beta}_{t}}^*}) \leq cn.$ Therefore, it is legitimate to use the online stopping rules with respect to estimated variances to guarantee a low regret bound.

\subsection{Online Stopping Algorithms}
According to the general stopping principles, if our parameter of interest is $\beta$, we need to find an unbiased estimator $\hat{\beta}_t$ and its estimated variance $S_t$, then we can devise the stopping rules based on $\lVert  S_t \rVert.$ To devise this procedure, we first review some corresponding results in linear contextual bandits.

\subsubsection{Sufficient Statistics in Contextual Bandits}
Suppose we have OLS estimators $\hat{\beta}_{t,1}^{OLS},\hat{\beta}_{t,0}^{OLS}$ for $\beta_1, \beta_0$ respectively in each batch, i.e.
$$\hat{\beta}_{t,1}^{OLS} = ( \sum_{i=1}^n 1_{A_{t,i}=1}X_{t,i}X_{t,i}^T)^{-1}\sum_{i=1}^n 1_{A_{t,i}=1}X_{t,i}^TY_{t,i}$$
$$\hat{\beta}_{t,0}^{OLS} = ( \sum_{i=1}^n 1_{A_{t,i}=0}X_{t,i}X_{t,i}^T)^{-1}\sum_{i=1}^n 1_{A_{t,i}=0}X_{t,i}^TY_{t,i}.$$
We also denote the contextual matrices under different arms as
$$\underline{X}_{t,1} =\sum_{i=1}^n 1_{A_{t,i}=1}X_{t,i}X_{t,i}^T \in R^{d \times d} $$
$$\underline{X}_{t,0} =\sum_{i=1}^n 1_{A_{t,i}=0}X_{t,i}X_{t,i}^T \in R^{d \times d} $$

In each batch, our probabilistic model is the distribution of action-reward pairs conditional on the previous history and observed contexts. In fact, we can use the above defined variables to derive the sufficient statistics in our linear contextual bandit model.

\begin{theorem} The sufficient statistics with respect to $A_{1:t,1:n},Y_{1:t,1:n} | \mathcal{F}_t $ is $$T_{1:t}= (\{ \hat{\beta}_{j,1}^{OLS},\underline{X}_{j,1},\hat{\beta}_{j,0}^{OLS},\underline{X}_{j,0} \}_{j=1}^t )    $$ \end{theorem}

\begin{proof} In batch t, the probabilistic model is the action-reward pairs $\{A_{t,i},Y_{t,i} \}_{i=1}^n$ conditional on observed contexts $X_{t,1:n}$ and current policy $\pi_t$, i.e.
$$p(A_{t,1:n},Y_{t,1:n} | X_{t,1:n},\pi_t )$$ 
$$= \prod_{i=1}^n \pi_t(A_{t,i})\prod_{i=1}^n \frac{1_{A_{t,i}=1} }{\sqrt{2\pi}\sigma_1} e^{-\frac{1_{A_{t,i}=1}(y_{t,i}-x_{t,i}^T\beta_1)^2}{2\sigma_1^2}}\prod_{i=1}^n \frac{1_{A_{t,i}=0}}{\sqrt{2\pi}\sigma_0} e^{-\frac{1_{A_t,i}=0(y_{t,i}-x_{t,i}^T\beta_0)^2}{2\sigma_0^2}}.$$
According to Factorization Theorem, the sufficient statistics for the above conditional model is $$T_t'=(\sum_{i=1}^n 1_{A_{t,i}=1}X_{t,i}X_{t,i}^T,\sum_{i=1}^n 1_{A_{t,i}=1}X_{t,i}^TY_{t,i},\sum_{i=1}^n 1_{A_{t,i}=0}X_{t,i}X_{t,i}^T,\sum_{i=1}^n 1_{A_{t,i}=0}X_{t,i}^TY_{t,i}).$$
Since $T_t=(\hat{\beta}_{t,1}^{OLS},\underline{X}_{t,1},\hat{\beta}_{t,0}^{OLS},\underline{X}_{t,0})$ is a one-to-one function of $T_t'$, it is also a sufficient statistics for $\beta_1,\beta_0$ in the conditional model at batch t.

For the full conditional model, we have
$$p(A_{1:t,1:n},Y_{1:t,1:n} | \mathcal{F}_t )$$
$$=\frac{1}{p(\mathcal{F}_t)} \prod_{j=1}^t p(A_{j,1:n},Y_{j,1:n} | X_{j,1:n},\pi_j ) .$$

Therefore, the sufficient statistics with respect to $A_{1:t,1:n},Y_{1:t,1:n} | \mathcal{F}_t $ is $$T_{1:t}= (\{ \hat{\beta}_{j,1}^{OLS},\underline{X}_{j,1},\hat{\beta}_{j,0}^{OLS},\underline{X}_{j,0} \}_{j=1}^t )  $$
\end{proof}

\subsubsection{Convergence results}
Under the two-arm setting, define the BOLS estimator for each batch $t \in [1 : T]$ as: $\hat{\beta}_t^{\text{BOLS}} =  (\hat{\beta}_{t,0},\hat{\beta}_{t,1})^T.$ We have the Lemma 3 and Lemma 4 according to results in \cite{zhang2020inference}.

\begin{lemma} Under Assumption 1,2,3 and a conditional clipping rate $f(n) = c$ for some $0 \leq c < \frac{1}{2}$, as the batch size $n \to \infty$, we have
\[
\begin{bmatrix}
\text{diag}[\underline{X}_{1,0}, \underline{X}_{1,1}]^{1/2} (\hat{\beta}_1^{\text{BOLS}} - \beta_1) \\
\text{diag}[\underline{X}_{2,0}, \underline{X}_{2,1}]^{1/2} (\hat{\beta}_2^{\text{BOLS}} - \beta_2) \\
\vdots \\
\text{diag}[\underline{X}_{T,0}, \underline{X}_{T,1}]^{1/2} (\hat{\beta}_T^{\text{BOLS}} - \beta_T)
\end{bmatrix}
\overset{D}{\to} \mathcal{N}(0, \sigma^2 I_{2Td}) .
\] \end{lemma}

\begin{lemma} Assuming the conditions of Lemma 3, for any batch $t \in [1 : T]$ and any arm $a \in \{0,1\}$, as $n \to \infty$, we have
\[
\left[
\sum_{i=1}^n \mathbb{I}_{A_{t,i} = a} X_{t,i} X_{t,i}^\top
\right] [n Z_{t,a} P_{t,a}]^{-1} \overset{P}{\to} I_d
\]
and
\[
\left[
\sum_{i=1}^n \mathbb{I}_{A_{t,i} = a} X_{t,i} X_{t,i}^\top
\right]^{1/2} [n Z_{t,a} P_{t,a}]^{-1/2} \overset{P}{\to} I_d
\]
where $P_{t,a} := \mathbb{P}(A_{t,i} = a | H_{t-1}), Z_{t,a} := \mathbb{E}[X_{t,i} X_{t,i}^\top | H_{t-1}, A_{t,i} = a]$. \end{lemma}

From Theorem 3 and Lemma 3, we see that if the adaptive policies $\{\pi_j\}_{j=1}^t$ can be totally determined by the sufficient statistics $\{ \hat{\beta}_{j,1}^{OLS},\underline{X}_{j,1},\allowbreak \hat{\beta}_{j,0}^{OLS},\underline{X}_{j,0} \}_{j=1}^t $, then we have
$$A_{t,1:n},Y_{t,1:n} | X_{t,1:n},\pi_t \sim   T_t| T_{1:t-1}, X_{t,1:n},$$
where $T_t=(\hat{\beta}_{t,1}^{OLS},\underline{X}_{t,1},\hat{\beta}_{t,0}^{OLS},\underline{X}_{t,0}).$ This scheme contains a large number of contextual bandit algorithms, including $\epsilon$-greedy algorithm, Upper Confidence Bound and Thompson sampling, since the mean estimates $ \{\hat{\beta}_{t,1}^{OLS}, \hat{\beta}_{t,0}^{OLS} \}$ and variance estimates $\{ \underline{X}_{t,1}^{-1},\underline{X}_{t,0}^{-1}\}$ are sufficient enough to determine the adaptive policies in them. Moreover, if the stopping time $I(T=t)$ is measurable with respect to the sufficient statistics $T_{1:t}$, then we can conduct the conditional inference totally based on $T_{1:t}$. If we define the new filtration $\mathcal{F'}_t 
 =\sigma(T_{1:t}) $, then  $A_{1:t,1:n},Y_{1:t,1:n} | \mathcal{F}_t \sim  T_{1:t}|\mathcal{F'}_t.$

\subsubsection{Online Estimators and Stopping rules}
Though Lemma 3 gives the joint distribution of the batched OLS estimators, it is still intractable to use it for inferences on $\beta$ because there are multiple batched estimators in it. We prefer to combine them and use low-dimensional statistics to make this tractable, though it may cause some loss in information compared to the sufficient statistics. In addition, we require the statistics to contain only the variables in the sufficient statistics $T_{1:t}$ so that the conditional inference after the experiment is explicit. Building on these considerations, we select a weighted combination of batched OLS estimators for online inference. Suppose that we want to estimate $\beta_1$, to construct the corresponding online estimators, we first state the following theorem.

\begin{theorem} In Batched Bandits, if the clipping probability $p_t \to 0$, we let $ \Sigma_1^*= E_{x\sim P_X}(xx^T| (\beta_1-\beta_0)^Tx>0 ).$ If the clipping probability $p_t \to p>0$, we let $\Sigma_1^* = (1-p)E_{x\sim P_X} (xx^T|(\beta_1-\beta_0)^Tx>0)+pE_{x\sim P_X} (xx^T|(\beta_1-\beta_0)^Tx<0).$ When $t\to \infty, n\to \infty$, we have 
\begin{equation} \frac{1}{\sqrt{nt}} \sum_{j=1}^t \underline{X}_{j,1} (\hat{\beta}_{j,1}^{OLS}-\beta_1) \xrightarrow{D} N(0, \Sigma_1^* \sigma^2  )    \end{equation} \end{theorem}

\begin{proof} We first consider the term $\frac{1}{\sqrt{n}}\underline{X}_{t,1} (\hat{\beta}_{t,1}^{OLS}-\beta_1)= \frac{1}{\sqrt{n}}  \sum_{i=1}^n   1_{A_{t,i}=1} X_{t,i}e_{t,i}.$ Observe that \\$ E( \frac{1}{\sqrt{n}} \sum_{i=1}^n  1_{A_{t,i}=1} X_{t,i}e_{t,i} |  H_{t-1} ) =0   $, thus (4.2.1) is the sum of a martingale difference sequence. We check the conditions in Martingale Central Limit Theorem to derive its asymptotic distribution.

\textbf{Conditional Variance.} Let $$\Sigma_t=E([ \frac{1}{\sqrt{n}} \sum_{i=1}^n  1_{A_{t,i}=1}X_{t,i}e_{t,i}] [ \frac{1}{\sqrt{n}}  \sum_{i=1}^n  1_{A_{t,i}=1} X_{t,i}e_{t,i}]^T            |H_{t-1} )    $$
$$= E( \frac{1}{n} \sum_{i=1}^n  1_{A_{t,i}=1}X_{t,i}X_{t,i}^T e_{t,i}^2| H_{t-1})   $$
$$=\sigma^2 Z_{t,1}P_{t,1}  \xrightarrow{P} \Sigma_1^* \sigma^2          $$
By Cesàro Mean Theorem, $\frac{1}{t} \sum_{i=1}^t \Sigma_i \xrightarrow{P} \Sigma_1^* \sigma^2   , t \to \infty.    $

\textbf{Lindeberg Condition.} Let $d_j =\frac{1}{\sqrt{n}}\sum_{i=1}^n  1_{A_{t,i}=1} X_{t,i}e_{t,i}$, then $\lVert d_j \rVert^2 =\frac{1}{n} \Vert \sum_{i=1}^n 1_{A_{t,i}=1} X_{t,i}e_{t,i} \rVert^2  \leq  L' $, due to the assumption that the contexts and error terms are uniformly bounded. Also, note that $e_{t,i}$ are sub-Gaussian variables, thus $  P( \lVert  1_{A_{t,i}=1} X_{t,i}e_{t,i}   \rVert \geq \epsilon |H_{t-1})     $ would decay exponentially. With above results, we have $ \sum_{j=1}^t E(\lVert \frac{d_j}{\sqrt{t}} \rVert^2 1_{\lVert d_j \rVert > \epsilon } |  H_{t-1}  ) \to 0, t \to \infty.$

\textbf{Bounded Variance Growth.} The trace of $\Sigma_t$ is uniformly bounded because the contexts are uniformly bounded.

Therefore, $\{ \frac{1}{\sqrt{n}}\underline{X}_{j,1} (\hat{\beta}_{j,1}^{OLS}-\beta_1)\}_{j=1}^t $ satisfies conditions in the Martingale Central Limit Theorem. As a result, we have
$$ \frac{1}{\sqrt{t}}  \sum_{j=1}^t     \frac{1}{\sqrt{n}}\underline{X}_{j,1} (\hat{\beta}_{j,1}^{OLS}-\beta_1)  \to N(0, \Sigma_1^* \sigma^2 ).$$ \end{proof}

\begin{corollary} Let $\hat{\beta}_{t,1}^{IVW} =( \sum_{j=1}^t \underline{X}_{t,1})^{-1}  ( \sum_{j=1}^t \underline{X}_{t,1}\hat{\beta}_{t,1}^{OLS})$ be the inverse-variance weighted estimator, then we have 
\begin{equation}  \sqrt{t} (\hat{\beta}_{t,1}^{IVW}-\beta_1) \xrightarrow{D} N(0,n(\Sigma_1^*)^{-1}\sigma^2   ). \end{equation}
The consistent estimator for the variance of $\hat{\beta}_{t,1}^{IVW}$ is $\hat{\Sigma}_{t,1}= 
n(\sum_{j=1}^t \underline{X}_{t,1})^{-1}\sigma^2.$   \end{corollary}

\begin{proof} We rewrite $\hat{\beta}_{t,1}^{IVW} $ as $\hat{\beta}_{t,1}^{IVW} =( \frac{1}{\sqrt{nt}} \sum_{j=1}^t \underline{X}_{t,1})^{-1}  (\frac{1}{\sqrt{nt}} \sum_{j=1}^t \underline{X}_{t,1}\hat{\beta}_{t,1}^{OLS}).$ Let ${\Sigma}_t' =(\frac{1}{\sqrt{nt}} \sum_{j=1}^t \underline{X}_{t,1})^{-1}.$ According to Theorem 4, we have $\frac{1}{\sqrt{nt}} \sum_{j=1}^t \underline{X}_{j,1} \hat{\beta}_{j,1}^{OLS}-   \frac{1}{\sqrt{nt}} \sum_{j=1}^t    \underline{X}_{j,1}\beta_1  \xrightarrow{D}  N(0,\Sigma_1^* \sigma^2 ).$ Also, note that $\frac{1}{\sqrt{nt}} \sum_{j=1}^t \underline{X}_{j,1}\beta_1 \xrightarrow{P} \sqrt{\frac{t}{n}}  \Sigma_1^* \beta_1 $, thus $ 
\frac{1}{\sqrt{nt}} \sum_{j=1}^t \underline{X}_{j,1} \hat{\beta}_{j,1}^{OLS} \xrightarrow{D}  N( \sqrt{ \frac{t}{n}} \Sigma_1^* \beta_1,\Sigma_1^* \sigma^2)  $ by Slutsky's Theorem. As $\frac{1}{\sqrt{nt}} \sum_{j=1}^t \underline{X}_{j,1} \xrightarrow{\| \cdot \|} \sqrt{\frac{t}{n} } \Sigma_1^*  $, and $\Sigma_1^*$ is invertible, we have ${\Sigma}_t'\xrightarrow{P} \sqrt{\frac{n}{t}} (\Sigma_1^*)^{-1}$. Using Slutsky's Theorem again, we get $$\sqrt{t} (\hat{\beta}_{t,1}^{IVW}-\beta_1 ) \xrightarrow{D} N(0, n(\Sigma_1^*)^{-1}\sigma^2   ) .$$ 
Also, note that $\hat{\Sigma}_{t,1}= 
n(\sum_{j=1}^t \underline{X}_{t,1})^{-1}\sigma^2 \xrightarrow{P} \frac{n}{t} (\Sigma_1^*)^{-1}\sigma^2  $, therefore it is a consistent estimator for the variance of $\hat{\beta}_{t,1}^{IVW} $.
\end{proof}

Similarly, we can construct the inverse-variance weighted estimator $ \hat{\beta}_{t,0}^{IVW}$ for $\beta_0$ and the corresponding variance estimator $ \hat{\Sigma}_{t,0}$. Following our proposed general principles of constructing online stopping rules, we can stop the experiment when the variances are small enough. The online stopping rules with threshold or opportunity cost based on the estimated variance of the statistics $\hat{\beta}_{t,1}^{IVW}$ and $ \hat{\beta}_{t,0}^{IVW}$ are:

\textbf{Online Stopping with Threshold.} Stop the experiment when $ \lVert \hat{\Sigma}_{t,1}  \rVert   \leq k , \lVert \hat{\Sigma}_{t,0}  \rVert   \leq k   .$

\textbf{Online Stopping with Opportunity Cost.} Stop the experiment when $ \lVert \hat{\Sigma}_{t-1,1}  \rVert- \lVert \hat{\Sigma}_{t,1}  \rVert \leq c', \lVert \hat{\Sigma}_{t-1,0}  \rVert- \lVert \hat{\Sigma}_{t,0}  \rVert \leq c' ,$ where $c'$ is the transformed opportunity cost in terms of variances.

\subsubsection{Conditional Inference Procedure}
The statistics we choose for constructing stopping rule and inference is $(\hat{\beta}_{t,1}^{IVW}, \hat{\beta}_{t,0}^{IVW}).$ After observing a realized stopping time T, we need to conduct conditional inference based on the probabilistic model $(\hat{\beta}_{T,1}^{IVW}, \hat{\beta}_{T,0}^{IVW} )| \pi_{1:T}, T .$ In fact, we can ignore the condition on $\pi_{1:T}$ when T is large, because Corollary 3 shows that the asymptotic distribution of $\hat{\beta}_{T,1}^{IVW}$ and $ \hat{\beta}_{T,0}^{IVW} $ does not depend on adaptive policies. Therefore, we can conduct our inference only conditional on the stopping time, i.e. $$  (\hat{\beta}_{T,1}^{IVW}, \hat{\beta}_{T,0}^{IVW} )| \pi_{1:T}, T \sim (\hat{\beta}_{T,1}^{IVW}, \hat{\beta}_{T,0}^{IVW} )|  T   .$$

Next, we devise our conditional inference procedure. In practice, we may need to estimate the variance of the error term $\sigma^2$. In heteroskedastic situations, we even need to estimate $E(e_t |a_t=1)=\sigma_1^2 $ and $ E(e_t |a_t=0)=\sigma_0^2$ respectively. To do this, we modify our estimated variance of the inverse-variance weighted estimators as follows. Let
\begin{equation} \hat{\Sigma}_{t,1}= 
n(\sum_{j=1}^t \underline{X}_{t,1})^{-1} \frac{\sum_{i=1}^t \sum_{j=1}^n I(a_{i,j}=1)(y_{i,j}- x_{i,j}^T \hat{\beta}_{t,1}^{IVW})^2  }{ \sum_{i=1}^t \sum_{j=1}^n I(a_{i,j}=1)  }  .  \end{equation}

\begin{equation}\hat{\Sigma}_{t,0}= 
n(\sum_{j=1}^t \underline{X}_{t,0})^{-1} \frac{\sum_{i=1}^t \sum_{j=1}^n I(a_{i,j}=0)(y_{i,j}- x_{i,j}^T \hat{\beta}_{t,1}^{IVW})^2  }{  \sum_{i=1}^t \sum_{j=1}^n I(a_{i,j}=0)   }  .  \end{equation}

Taking the threshold stopping rule for example, we stop the experiment when $\hat{\Sigma}_{t,1},\hat{\Sigma}_{t,0}$ are small, i.e., 
$$
T = \min \{t : \| \hat{\Sigma}_{t,1} \| \leq p, \| \hat{\Sigma}_{t,0} \| \leq p  \}.$$

Next, we establish the testing procedure in our conditional inference model. Suppose that we want to test $H_0: \beta_1=\beta_1^*,\beta_0=\beta_0^*$ vs $H_1$: not equal. We can do the following steps:

\textbf{Step 1.} Sample $\tilde{\beta}_1$, $ \tilde{\beta}_0 $ from asymptotic distribution of $ \hat{\beta}_{T,1}^{IVW}$, $\hat{\beta}_{T,0}^{IVW} $  , with the constraint that the experiment should not stop until the realized stopping time, i.e. $\lVert \hat{\Sigma}_{t,1}\rVert \leq p, \lVert \hat{X}_{t,0}^{-1}\rVert \leq p$ iff. $t=T$, where T is the observed stopping time.

\textbf{Step 2.} Bootstrap a series of samples $\tilde{\beta}_1$, $ \tilde{\beta}_0 $. Reject $H_0$ if $\beta_1^*,\beta_0^*$ do not lie in the empirical confidence interval. 

By selecting a single online estimator instead of a series of batched estimators, we maintain moderate computational complexity. Since deriving a closed-form solution for conditional inference in the contextual bandit setting is challenging, our primary and future focus should be on improving sampling and computational efficiency.

\bibliographystyle{plain} 
\bibliography{mybib}

\end{document}